\documentclass[11pt]{article}

\usepackage[a4paper,margin=1in]{geometry}
\usepackage{amsmath,amssymb,amsthm}
\usepackage{booktabs}
\usepackage{graphicx}
\usepackage{hyperref}
\newtheorem{theorem}{Theorem}[section]
\newtheorem{lemma}[theorem]{Lemma}
\newtheorem{corollary}[theorem]{Corollary}
\newtheorem{remark}[theorem]{Remark}

\usepackage{tikz}
\usepackage{float}
\usetikzlibrary{arrows.meta,decorations.pathreplacing}
\title{Parallel Machine Scheduling with a Single Server and Loading--Unloading Operations}

\author{%
Keramat Hasani  \\[1ex]
Centre for Maritime Studies, National University of Singapore, Singapore  \\[1ex]
email: hasani@nus.edu.sg  \\[3ex]
  Frank Werner  \\[1ex]
Otto-von-Guericke Business School Magdeburg, Germany \\[1ex]
  email: frank.werner@ovgu.de; ORCID: 0000-0002-0709-3591
}
 
\vspace{2mm}

\date{August 11, 2026}
 
\begin{document}

\maketitle


\begin{center}
\textbf{Abstract}
\end{center}

This paper investigates a parallel machine scheduling problem featuring a single common server responsible for both loading and unloading operations. Each job consists of a unit-time loading operation, non-preemptive processing on one of \(m\) identical machines, and a unit-time unloading operation executed by the same server. The objective is to minimize the makespan. Unlike classical loading-only common-server models, our setting requires the server to handle an unloading operation immediately after a job's processing phase concludes. We prove that the corresponding decision problem is strongly NP-complete when the number of machines is given as part of the input. Furthermore, we analyse the worst-case performance of the List Scheduling (LS) and Longest Processing Time (LPT) heuristics. For three machines, we establish that Algorithm LS achieves an approximation ratio of at most \(5/2\). For an arbitrary fixed \(m \ge 3\), we show that the general LS bound approaches \(4-3/m\) as the number of jobs grows, whereas for Algorithm LPT, we prove the finite-instance ratio
$3-\frac{2}{m}+\frac{(m-1)(m-2)}{mn}.$ Thus, for each fixed \(m\), the LPT bound approaches \(3-2/m\) as the number of jobs grows.\\[1ex]


\section{Introduction}
\label{sec:introduction}

Many production and transport systems contain a small number of expensive
resources that must be shared by several parallel service units. In flexible
manufacturing systems, for example, machines may be able to process jobs in
parallel, while a single robot or handling device performs tool changes, part
setups, loading operations, or transfers between pieces of equipment. In such
systems, machine capacity alone does not determine throughput; the availability
of the common handling resource can create forced idle time and can change the
structure of good schedules. Similar bottlenecks arise in transport operations.
In ports, vessels may be assigned to different berths, but their movements may
have to pass through a restricted access channel or a one-way navigation channel.
Integrated berth and channel-planning models explicitly recognise that vessel
service decisions and channel-use decisions interact, and that channel capacity
can become the operational bottleneck \cite{ZhenLiangZhugeLeeChew2017,
LiuLiShengWang2021,ZhangZheng2020}. These applications motivate scheduling
models in which several parallel processors share a single auxiliary resource.

This paper studies one such model. We consider the problem
$
        P,S1\mid s_j=t_j=1\mid C_{\max}.
$
There are \(m\) identical parallel machines and one common server. Each job must
first be loaded by the server, then processed without interruption on one
machine, and then immediately unloaded by the same server. Here, the loading and unloading
times are both equal to one. 

The loading-only common-server problem has been studied for several decades.
In that model, each job requires a setup, or loading operation, by a common
server immediately before processing, after which the server is released.
Hall et al. \cite{HallPottsSriskandarajah2000} developed a
systematic treatment of this environment and classified a number of polynomial,
pseudo-polynomial and NP-hard cases. Kravchenko and Werner
\cite{KravchenkoWerner1997} proved, among other results, that the variable
machine problem
$P,S1\mid s_j=1\mid C_{\max}$
is strongly NP-hard. Their proof uses a modular construction from
\textsc{3-Partition}. Brucker et al.
\cite{BruckerDhaenensFlipoKnustKravchenkoWerner2002} continued the complexity
classification of parallel machine problems with a single server. These papers
show that adding even one common loading server to a parallel machine system
substantially changes the complexity landscape. The broader literature on setup
times and setup costs is surveyed by Allahverdi et al.
\cite{AllahverdiNgChengKovalyov2008}.

The model in the present paper is different from the loading-only setting
because the server must return to each job at the completion of its processing.
From the server's point of view, every job is therefore a coupled pair of unit
operations with an exact delay between them. This places the problem close to
coupled-task scheduling with exact delays, a class surveyed by Khatami et al. \cite{KhatamiSalehipourCheng2020}. However, our problem is
not a pure coupled-task problem: while the two server operations are separated
by an exact delay, the job also occupies one of the parallel machines throughout
the interval from loading to unloading. The difficulty is therefore caused by
the simultaneous interaction of machine capacity and server feasibility.

The closest theoretical predecessor is the two-machine loading-unloading model
studied by Jiang et al. \cite{JiangZhangHuDongJi2015}. They considered two
identical parallel machines sharing a single server that loads and unloads all
jobs, with unit loading and unloading times, and analysed the worst-case
performance of the algorithms LS (list scheduling) and LPT (longest processing time) 
for the makespan objective. A preemptive variant was
studied by Jiang et al. \cite{JiangWangZhou2014}, who gave an
\(O(n\log n)\) algorithm. More recently, Elidrissi et al.
\cite{ElidrissiBenmansourHasaniWerner2024} developed exact formulations,
polynomial cases, lower bounds and computational methods for the nonpreemptive
two-machine common-server loading-unloading problem. These works provide a
foundation for the two-machine case. Much less is known when the number of
machines is variable or when one asks how the two-machine structural arguments
change when more than two machines are available.

This paper addresses the above gap for the unit loading-unloading model. We first
settle the complexity of the variable-machine case by proving that the decision
problem is strongly NP-complete when \(m\) is part of the input. We then study the
worst-case performance of list-based schedules. For three machines, we prove
\[
        C_{\max}^{\rm LS}\le \frac52 C^* ,
\]
where $C^*$ denotes the optimal makespan. 
Since LPT is a particular list schedule, the same bound also holds for LPT when
\(m=3\). For arbitrary \(m\ge3\), we establish the general LS guarantee
\[
        C_{\max}^{\rm LS}
        \le
        \left(
        4-\frac3m-\frac{2(m-1)}{mn}
        \right) C^* .
\]
We also prove the following bound for an LPT schedule:
\[
        C_{\max}^{\rm LPT}
        \le
        \left(
        3-\frac2m+\frac{(m-1)(m-2)}{mn}
        \right)C^*.
\]
These results show that the three-machine case has a stronger structure than
the general \(m\)-machine case, while the general estimates remain valid for any
number of machines.

The remainder of the paper is organised as follows.
Section~\ref{sec:problem-lower-bounds} defines the model and gives the lower
bounds used throughout. Section~\ref{sec:complexity-variable-m} proves the
strong NP-completeness result for the variable-machine decision problem.
Section~\ref{sec:server-blocking-preliminaries} introduces the blocking
notation. Section~\ref{sec:ls-three-machines-five-halves} proves the
three-machine \(5/2\)-bound for LS. Section~\ref{sec:ls-general-performance}
establishes the general-\(m\) LS bound. Section~\ref{sec:lpt-general-m} treats
the LPT rule.

\section{Problem Definition and Lower Bounds}
\label{sec:problem-lower-bounds}

We consider the integral-time scheduling problem
$ P,S1\mid s_j=t_j=1\mid C_{\max}.$
There are \(m\) identical parallel machines and one common server. Each job
\(J_j\) requires one unit of loading by the server, followed by an uninterrupted
processing on one machine, and finally one unit of unloading by the same server.
We write $e_j=p_j+2$ for the execution length of \(J_j\), namely the time from the beginning of
loading to completion. Since \(p_j\ge1\), we have \(e_j\ge3\). If \(J_j\) is
loaded at integer time \(S_j\), then its unloading slot $U_j$ and completion time $C_j$ are
\[
        U_j=S_j+e_j-1,
        \qquad
        C_j=S_j+e_j .
\]
A server operation at integer time \(t\) occupies the unit slot \([t,t+1)\).
Hence the server slots
\[
        S_j,\quad U_j,\qquad j=1,\ldots,n,
\]
are pairwise distinct. Job \(J_j\) occupies its assigned machine during the interval $[S_j,C_j).$ Figure~\ref{fig:job-structure} summarises this notation.
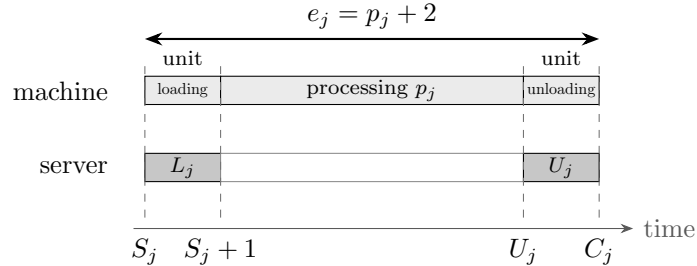
\begin{figure}[H]
\centering
\begin{tikzpicture}[
    x=1.0cm,
    y=0.85cm,
    >=Stealth,
    font=\small,
    machine/.style={draw=black, fill=black!8, minimum height=0.42cm, inner sep=0pt},
    server/.style={draw=black, fill=black!22, minimum height=0.42cm, inner sep=0pt},
    empty/.style={draw=black!45, fill=white, minimum height=0.42cm, inner sep=0pt}
]


\node[anchor=east] at (-0.35,1.20) {machine};
\node[anchor=east] at (-0.35,0.00) {server};

\draw[machine] (0,0.98) rectangle (1,1.42);
\draw[machine] (1,0.98) rectangle (5,1.42);
\draw[machine] (5,0.98) rectangle (6,1.42);

\node[scale=0.6] at (0.5,1.20) {loading};
\node at (3.0,1.20) {\scriptsize processing \(p_j\)};
\node[scale=0.6] at (5.5,1.20) {unloading};

\draw[server] (0,-0.22) rectangle (1,0.22);
\draw[empty]  (1,-0.22) rectangle (5,0.22);
\draw[server] (5,-0.22) rectangle (6,0.22);

\node at (0.5,0.00) {\scriptsize \(L_j\)};
\node at (5.5,0.00) {\scriptsize \(U_j\)};

\draw[->,black!65] (-0.15,-0.95) -- (6.45,-0.95) node[right] {time};

\foreach \x in {0,1,5,6}{
    \draw[black!55,dashed] (\x,-0.78) -- (\x,1.60);
}

\node[below] at (0,-0.95) {\(S_j\)};
\node[below] at (1,-0.95) {\(S_j+1\)};
\node[below] at (5,-0.95) {\(U_j\)};
\node[below] at (6,-0.95) {\(C_j\)};

\draw[<->,thick] (0,2) -- (6,2);
\node[above] at (3,2) {\(e_j=p_j+2\)};



\node[above] at (0.5,1.43) {\scriptsize unit};
\node[above] at (5.5,1.43) {\scriptsize unit};

\end{tikzpicture}

\caption{Timing of a job \(J_j\). }
\label{fig:job-structure}
\end{figure}

Let
$ E=\sum_{j=1}^{n}e_j,$ and $e_{\max}=\max_{1\le j\le n}e_j.$ The objective is to minimize
$C_{\max}=\max_{1\le j\le n}C_j.$ The optimal makespan is denoted by \(C^*\).

We shall use the following elementary lower bounds.

\begin{lemma}
\label{lem:lower-bounds}
Every feasible schedule satisfies
\[
        C^*\ge \frac{E}{m},
        \qquad
        C^*\ge 2n,
        \qquad
        C^*\ge e_{\max},
        \qquad
        C^*\ge \frac{E+2(m-1)}{m}.
\]
\end{lemma}

\begin{proof}
The first inequality is the machine-workload bound. The second follows from the
\(2n\) unit operations that must be performed by the common server. The third is
immediate from the largest execution length.

For the final inequality, shift the schedule so that the first server operation
starts at time zero. In the first slot \([0,1)\), at most one machine can be
occupied, since at most one job can be loaded. Thus, at least \(m-1\) units of
machine capacity are idle in this slot. In the final slot \([C^*-1,C^*)\), any
active job must complete at time \(C^*\), and hence must use the server for
unloading in slot \(C^*-1\). Since only one unloading can be performed in that
slot, at most one machine is occupied, and again at least \(m-1\) units of
machine capacity are idle.

Because \(e_j\ge3\) and \(n\ge1\), the first and final slots are distinct.
Therefore, we get
\[
        E\le mC^*-2(m-1),
\]
which gives
\[
        C^*\ge \frac{E+2(m-1)}{m}.
\]
\end{proof}


\section{Complexity of the Variable-Machine Problem}
\label{sec:complexity-variable-m}

We next consider the computational complexity of the makespan problem. The
result in this section concerns the case when the number of machines is part of
the input. The proof follows the modular idea used by Kravchenko and Werner for
the loading-only common-server problem, but the construction has to be modified
because each job now has two server operations, one for loading and one for
unloading.

Consider the decision version of
$ P,S1\mid s_j=t_j=1\mid C_{\max}.$
An instance asks whether there exists a feasible integral schedule with
$C_{\max}\le T .$

\begin{theorem}
\label{thm:variable-m-strongly-np-complete}
When \(m\) is part of the input, the decision problem
\[
        P,S1\mid s_j=t_j=1\mid C_{\max}
\]
is strongly NP-complete.
\end{theorem}

\begin{proof}
Membership in NP is immediate: a certificate specifies a machine assignment and
integer loading time for each job, and feasibility together with the bound
\(C_{\max}\le T\) can be verified in polynomial time.

The reduction is from \textsc{3-Partition}. An instance of \textsc{3-Partition} is a
set of \(3m\) positive integers
\[
        a_1,\ldots,a_{3m}
\]
such that
\[
        \sum_{i=1}^{3m}a_i=mB,
        \qquad
        \frac{B}{4}<a_i<\frac{B}{2}
        \quad (i=1,\ldots,3m).
\]
The question is whether the items can be partitioned into \(m\) disjoint triples
whose sums are all equal to \(B\). We may assume \(m\ge2\).

From this instance, we construct an instance of the loading-unloading scheduling
problem with \(m\) machines and \(3m\) jobs. For every item \(a_i\), create one
job \(J_i\) with execution length
\[
        e_i=2m a_i .
\]
Since \(e_i=p_i+2\), set
\[
        p_i=2m a_i-2 .
\]
As \(m\ge2\) and \(a_i\ge1\), these processing times are positive integers.
Finally, set the target makespan to
\[
        T=2mB+2m-2 .
\]

We show that the \textsc{3-Partition} instance is feasible if and only if the
constructed scheduling instance admits a schedule with makespan at most \(T\).

First, suppose that the items can be partitioned into triples
\[
        A_1,\ldots,A_m,
        \qquad
        \sum_{a_i\in A_j}a_i=B
        \quad (j=1,\ldots,m).
\]
For each machine \(j\), schedule the three jobs corresponding to \(A_j\)
consecutively, starting at time
\[
        L_j=2(j-1).
\]
If the three jobs on machine \(j\) are denoted by
\(J_{j,1},J_{j,2},J_{j,3}\), set
\[
        S_{j,1}=L_j,
        \qquad
        S_{j,h+1}=S_{j,h}+e_{j,h},
        \quad h=1,2.
\]
Thus, the jobs on each machine are processed without overlap. The total execution
length assigned to machine \(j\) is
\[
        \sum_{a_i\in A_j}2m a_i=2mB.
\]
Hence machine \(j\) completes at
\[
        L_j+2mB=2(j-1)+2mB.
\]
The largest completion time occurs on machine \(m\), and it is
\[
        2(m-1)+2mB=2mB+2m-2=T.
\]

It remains only to verify the server constraint. Every execution length is a
multiple of \(2m\). Hence all loading times on machine \(j\) are congruent to
\[
        2(j-1) \pmod{2m}.
\]
The unloading time of a job loaded at time \(S\) is
\[
        U=S+e_i-1.
\]
Therefore, all unloading times on machine \(j\) are congruent to
\[
        2(j-1)-1 \pmod{2m}.
\]
Thus, machine \(j\) uses one residue class modulo \(2m\) for its loadings and one
different residue class for its unloadings. Over all \(j=1,\ldots,m\), the
loading residues are
\[
        0,2,4,\ldots,2m-2,
\]
and the unloading residues are
\[
        2m-1,1,3,\ldots,2m-3.
\]
Together these are exactly the \(2m\) residue classes modulo \(2m\). Consequently
no two server operations coincide. This excludes loading-loading,
unloading-unloading, and loading-unloading collisions. The constructed schedule
is therefore feasible and has the makespan \(T\).

Conversely, suppose that the constructed scheduling instance has a feasible
schedule with
\[
        C_{\max}\le T.
\]
We prove that the original \textsc{3-Partition} instance has a feasible
partition.

No machine can process four jobs. Indeed, for any four jobs corresponding to
items \(a_{i_1},a_{i_2},a_{i_3},a_{i_4}\), we have
\[
        a_{i_1}+a_{i_2}+a_{i_3}+a_{i_4}>B.
\]
Since the item sizes are integral, this implies
\[
        a_{i_1}+a_{i_2}+a_{i_3}+a_{i_4}\ge B+1.
\]
The total execution length of the corresponding four jobs is therefore at least
\[
        2m(B+1)=2mB+2m.
\]
However, 
\[
        2mB+2m>T=2mB+2m-2.
\]
Thus, four jobs cannot be assigned to the same machine in any schedule with a 
makespan at most \(T\).

There are \(3m\) jobs and \(m\) machines. Since every machine processes at most
three jobs, every machine must process exactly three jobs.

Next, consider the total execution length assigned to any one machine. Each job
has an execution length divisible by \(2m\), and hence each machine load is a
multiple of \(2m\). No machine can have a load greater than \(2mB\). If a machine
load were larger than \(2mB\), then, being a multiple of \(2m\), it would be at
least
\[
        2m(B+1)=2mB+2m>T,
\]
which is impossible.

The total execution length of all jobs is
\[
        \sum_{i=1}^{3m} e_i
        =
        \sum_{i=1}^{3m}2m a_i
        =
        2m\cdot mB .
\]
Since there are \(m\) machines and each machine has a load at most \(2mB\), every
machine must have a load exactly equal to \(2mB\).

Therefore, each machine processes exactly three jobs, and the three jobs assigned
to any machine have total execution length \(2mB\). If these jobs correspond to
items \(a_i,a_j,a_k\), then
\[
        2m(a_i+a_j+a_k)=2mB,
\]
and hence
\[
        a_i+a_j+a_k=B.
\]
Therefore, the triples induced by the \(m\) machine assignments form a feasible
3-partition.

We have shown that the constructed scheduling instance has a schedule with a 
makespan at most \(T\) if and only if the original \textsc{3-Partition} instance
is feasible.

The transformation is polynomial, and all numerical values created by the
reduction are polynomially bounded in the numerical data of the
\textsc{3-Partition} instance. Since \textsc{3-Partition} is strongly
NP-complete, the reduction proves strong NP-hardness. Together with membership
in NP, the decision problem is strongly NP-complete.
\end{proof}


\section{Server Blocking Preliminaries}
\label{sec:server-blocking-preliminaries}

After the first \(k\) list insertions, the jobs
\(J_1,\ldots,J_k\) form the current LS prefix; we refer to them as prefix jobs.
For a partial schedule, let
$B=\{S_j,U_j:\ J_j \text{ is a prefix job}\}$
be the set of occupied server slots. Consider inserting a job \(J\) with
execution length \(e\). We call
\[
        d=e-1
\]
the loading-to-unloading offset of \(J\), since its unloading slot is exactly
\(d\) time units after its loading slot.
If a scheduling rule selects a machine whose current availability time is \(a\),
then a candidate loading time is an integer \(t\ge a\). The candidate is
server-feasible precisely when
\[
        t\notin B,
        \qquad
        t+d\notin B.
\]
List scheduling (LS) selects a machine with minimum current completion time and
chooses the earliest server-feasible candidate on that machine.

If a rejected candidate \(t\) satisfies \(t\in B\), we say that it has a
direct blocker. If it satisfies \(t+d\in B\), we say that it has a shifted
blocker. A shifted blocker is an already scheduled server operation at time
\(t+d\); it may be either a loading operation or an unloading operation.

Let \(S\) be the loading time assigned to the inserted job. The failed candidate
interval is
\[
        [a,S)_{\mathbb Z}=\{a,a+1,\ldots,S-1\}.
\]
If \(S=a\), this interval is empty. Figure~\ref{fig:blocking-concept} illustrates these notions.

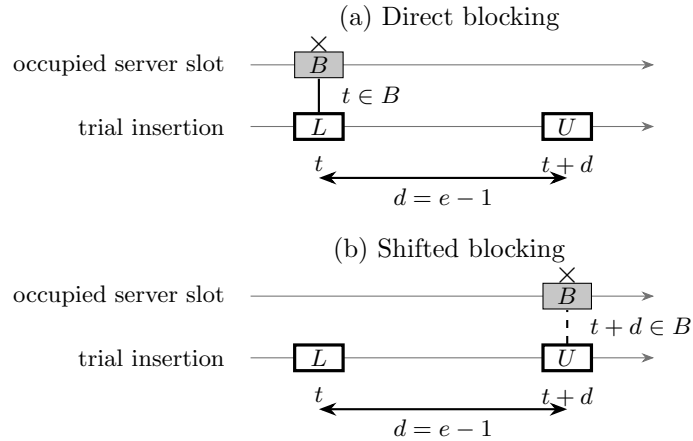
\begin{figure}[H]
\centering
\begin{tikzpicture}[
    x=0.82cm,
    y=0.68cm,
    >=Stealth,
    font=\footnotesize,
    op/.style={draw=black, minimum width=0.64cm, minimum height=0.34cm, inner sep=0pt},
    old/.style={op, fill=black!22},
    new/.style={op, fill=white, very thick},
    note/.style={fill=white, inner sep=1pt}
]

\begin{scope}[shift={(0,4.5)}]
\node[font=\small] at (3.1,2.45) {(a) Direct blocking};

\draw[->,black!55] (-0.1,1.55) -- (6.4,1.55);
\draw[->,black!55] (-0.1,0.35) -- (6.4,0.35);

\node[anchor=east] at (-0.35,1.55) {occupied server slot};
\node[anchor=east] at (-0.35,0.35) {trial insertion};

\node[old] at (1.0,1.55) {\(B\)};
\node[new] at (1.0,0.35) {\(L\)};
\node[new] at (5.0,0.35) {\(U\)};

\node[below=1pt] at (1.0,0.03) {\(t\)};
\node[below=1pt] at (5.0,0.03) {\(t+d\)};

\draw[thick] (1.0,0.62) -- (1.0,1.28);
\node[font=\large] at (1.0,1.96) {\(\times\)};
\node[note,anchor=west] at (1.32,0.95) {\(t\in B\)};

\draw[<->,thick] (1.0,-0.65) -- (5.0,-0.65);
\node[below=2pt] at (3.0,-0.55) {\(d=e-1\)};
\end{scope}

\begin{scope}[shift={(0,0)}]
\node[font=\small] at (3.1,2.45) {(b) Shifted blocking};

\draw[->,black!55] (-0.1,1.55) -- (6.4,1.55);
\draw[->,black!55] (-0.1,0.35) -- (6.4,0.35);

\node[anchor=east] at (-0.35,1.55) {occupied server slot};
\node[anchor=east] at (-0.35,0.35) {trial insertion};

\node[old] at (5.0,1.55) {\(B\)};
\node[new] at (1.0,0.35) {\(L\)};
\node[new] at (5.0,0.35) {\(U\)};

\node[below=1pt] at (1.0,0.03) {\(t\)};
\node[below=1pt] at (5.0,0.03) {\(t+d\)};

\draw[thick,dashed] (5.0,0.62) -- (5.0,1.28);
\node[font=\large] at (5.0,1.96) {\(\times\)};
\node[note,anchor=west] at (5.32,0.95) {\(t+d\in B\)};

\draw[<->,thick] (1.0,-0.65) -- (5.0,-0.65);
\node[below=2pt] at (3.0,-0.55) {\(d=e-1\)};
\end{scope}

\end{tikzpicture}

\caption{Direct and shifted blocking of a candidate loading time. }
\label{fig:blocking-concept}
\end{figure}

In direct blocking, the candidate loading slot is already occupied, \(t\in B\). In shifted
blocking, the corresponding unloading slot is already occupied, \(t+d\in B\).
Thus, a rejected candidate belongs to \(B\cup(B-d)\).

\begin{lemma}
\label{lem:blocking-cover}
For every LS insertion, we have
\[
        [a,S)_{\mathbb Z}\subseteq B\cup(B-d),
\]
where
\[
        B-d=\{t-d:\ t\in B\}.
\]
\end{lemma}

\begin{proof}
Let \(t\in[a,S)_{\mathbb Z}\). Since \(S\) is the earliest feasible candidate,
\(t\) is infeasible. Hence either \(t\in B\), or \(t+d\in B\), or both. This is
equivalent to \(t\in B\cup(B-d)\).
\end{proof}

The next observation is a simple consequence of the LS rule.

\begin{lemma}
\label{lem:ls-terminal-property-prelim}
Consider an LS prefix and let
$a=\min_i c_i$
be the minimum current machine completion time. If a prefix job has a machine
occupation slot or a server operation at some time \(t\ge a\), then that job is
terminal on its machine in the current prefix.
\end{lemma}

\begin{proof}
Choose a machine whose current completion time in the prefix is \(a\). During
the construction of the prefix, the completion time of this machine never
exceeded its final prefix value \(a\). Hence, before each earlier insertion,
there was at least one machine with completion time at most \(a\). Since LS
always selects a machine with minimum current completion time, no earlier job
could have been inserted with selected-machine availability strictly larger than
\(a\).

Now suppose that a prefix job \(J\) has a machine occupation slot or a server
operation at some time \(t\ge a\), but is not terminal on its machine. Let \(J'\)
be its successor on the same machine. Then \(C_J>a\), so \(J'\) was inserted
with machine availability \(C_J>a\), contradicting the preceding paragraph.
\end{proof}

For an integer slot \(t\), let \(M(t)\) denote the number of jobs occupying
machines during the interval \([t,t+1)\), and let
\[
        \chi(t)=
        \begin{cases}
        1, & \text{if the server is busy in slot }t,\\
        0, & \text{otherwise}.
        \end{cases}
\]
Define the aggregate resource profile
\[
        R(t)=M(t)+\chi(t).
\]
For a complete schedule, we have
\[
        \sum_t R(t)=E+2n,
\]
because each job contributes \(e_j\) machine-occupation units and two
server-operation units.



\section{A \(5/2\)-Approximation Bound for Algorithm LS on Three Machines}
\label{sec:ls-three-machines-five-halves}

We now specialise the preceding notation and blocking preliminaries to the case
\(m=3\). We use the integral-time convention, the offset notation
\(d_j=e_j-1\ge2\), and the aggregate resource profile
\[
        R(t)=M(t)+\chi(t)
\]
defined in Section~\ref{sec:server-blocking-preliminaries}. For a complete
schedule, we have 
\[
        \sum_t R(t)=E+2n .
\]

The proof relies on a structural property specific to the three-machine case:
after shifting the minimum current machine completion time to zero, at most two
previously scheduled terminal jobs can contribute relevant non-negative operations.
\subsection{Two-terminal structure}

Consider an LS prefix on three machines. Let
\[
        a=\min_i c_i
\]
be the minimum current machine completion time and shift time so that \(a=0\).
By Lemma~\ref{lem:ls-terminal-property-prelim}, any prefix job with a
machine-occupation slot or a server operation at a non-negative time is terminal
on its machine. Since one machine has completion time zero, at most two
previously scheduled terminal jobs can contribute such a non-negative occupation or server operation.

If such terminal jobs exist, denote them by \(X\) and \(Y\), ordered by their LS
insertion order. We write
\[
        L_X,\quad U_X=L_X+d_X,
        \qquad
        L_Y,\quad U_Y=L_Y+d_Y,
\]
where \(d_X,d_Y\ge2\).

\begin{lemma}[Two-terminal endpoint structure]
\label{lem:two-terminal-endpoint-three}
In the shifted prefix described above, we have
\[
        L_X\le0\le U_X .
\]
If \(Y\) exists and \(L_Y>0\), then
\[
        L_Y\le2 .
\]
Moreover,
\[
        L_Y=2
        \quad\Longrightarrow\quad
        L_X=0,\qquad U_X=d_Y+1 .
\]
\end{lemma}

\begin{proof}
Since \(X\) has a non-negative machine-occupation slot or server operation, we
have \(U_X\ge0\). Suppose, for contradiction, that \(L_X>0\). When \(X\) was
inserted, its selected-machine availability was at most zero by the same 
argument for Algorithm LS as in Lemma~\ref{lem:ls-terminal-property-prelim}. Moreover, by the
choice of \(X\), no previously scheduled terminal job has a non-negative server operation.
Thus, neither the candidate loading slot \(0\) nor the candidate unloading slot
\(d_X\) was occupied by an earlier job. Loading \(X\) at time \(0\) would have
been server-feasible, contradicting the earliest-feasible choice of LS. Hence
\(L_X\le0\) and therefore, \(L_X\le0\le U_X\).

Now suppose that \(Y\) exists and \(L_Y>0\). When \(Y\) was inserted, its
selected-machine availability was at most zero. Every non-negative candidate
\[
        0,1,\ldots,L_Y-1
\]
was therefore blocked by an earlier server endpoint. By
Lemma~\ref{lem:ls-terminal-property-prelim}, the only previously scheduled job with possible
non-negative endpoints is \(X\). Hence a non-negative candidate \(\tau\) can be
blocked only if
\[
        \tau\in\{L_X,U_X\}
        \qquad\text{or}\qquad
        \tau+d_Y\in\{L_X,U_X\}.
\]
Since \(L_X\le0\), the possible non-negative blocking positions are
\[
        0 \quad\text{if }L_X=0,\qquad U_X,\qquad U_X-d_Y .
\]
The last two positions differ by \(d_Y\ge2\). Therefore, the three consecutive
candidates \(0,1,2\) cannot all be blocked, and so \(L_Y\le2\).

Finally assume \(L_Y=2\). Candidates \(0\) and \(1\) were both blocked when
\(Y\) was inserted. Candidate \(1\) cannot be blocked by \(L_X\), since
\(L_X\le0\). If candidate \(1\) were blocked directly by \(U_X=1\), then
\(U_X-d_Y\le -1\). Candidate \(0\) would then have to be blocked by \(L_X=0\),
which would imply \(d_X=U_X-L_X=1\), contradicting \(d_X\ge2\). Hence candidate
\(1\) is shifted-blocked by
\[
        U_X-d_Y=1,
\]
and so \(U_X=d_Y+1\). With this value of \(U_X\), candidate \(0\) cannot be
blocked by \(U_X\) or by \(U_X-d_Y\). Thus, candidate \(0\) must be blocked by
\(L_X=0\).
\end{proof}

Figure~\ref{fig:three-machine-boundary-case} illustrates only the boundary case
\(L_Y=2\) in the last part of Lemma~\ref{lem:two-terminal-endpoint-three}. In
this boundary case, the failed candidates before \(Y\) is loaded are \(0\) and
\(1\); the time \(2=L_Y\) is the selected loading time of \(Y\). Before \(Y\) is loaded, candidates \(t=0\) and \(t=1\) must be blocked.
Candidate \(t=0\) is directly blocked by \(L_X=0\), while candidate \(t=1\) is
shift-blocked by \(U_X\). Therefore, \(1+d_Y=U_X\), equivalently
\(U_X=d_Y+1\).

\begin{figure}[H]
\centering
\begin{tikzpicture}[
    x=1.2cm,
    y=1.0cm,
    >=Stealth,
    font=\sffamily,
    machine/.style={draw=black!70, fill=black!5, minimum height=0.6cm, inner sep=0pt, rounded corners=1pt},
    activejob/.style={draw=black!80, fill=blue!15, minimum height=0.5cm, inner sep=0pt},
    endpoint/.style={draw=black, fill=red!70, circle, minimum size=0.15cm, inner sep=0pt},
    candidate/.style={draw=black!60, fill=black!5, rectangle, minimum width=0.8cm, minimum height=0.4cm, inner sep=0pt},
    blocked/.style={draw=red!80, fill=red!10, rectangle, minimum width=0.8cm, minimum height=0.4cm, inner sep=0pt},
    guide/.style={dashed, black!35},
    axis/.style={->, thick, black!70},
    blocklabel/.style={fill=white, inner sep=1.5pt, font=\scriptsize, text=red!80!black}
]

\draw[axis] (-1, -1.5) -- (8, -1.5) node[right, text=black] {time \(t\)};
\foreach \x in {0,1,2,3,4,5,6,7} {
    \draw (\x, -1.4) -- (\x, -1.6);
    \node[below, font=\small] at (\x, -1.6) {\x};
    \draw[guide] (\x, -1.2) -- (\x, 3.5);
}

\node[anchor=east, font=\bfseries] at (-1, 3) {Machine 1};
\draw[machine] (-1, 2.7) rectangle (0, 3.3);
\node at (-0.5, 3) {Prefix};
\node[anchor=south, font=\small, text=blue!80!black] at (0, 3.3) {Min available (\(a=0\))};
\draw[thick, blue!80!black] (0, 2.6) -- (0, 3.4);

\node[anchor=east, font=\bfseries] at (-1, 1.8) {Machine 2};
\draw[activejob] (-1, 1.55) rectangle (4, 2.05);

\node[endpoint] (LX) at (0, 1.95) {};
\node[endpoint] (UX) at (4, 1.8) {};
\node[above left, font=\small] at (LX) {\(L_X=0\)};
\node[above right, font=\small] at (UX) {\(U_X=d_Y+1\)};

\node[anchor=east, font=\bfseries] at (-1, 0.6) {Machine 3};
\draw[activejob] (2, 0.35) rectangle (5, 0.85);

\node[endpoint] (LY) at (2, 0.6) {};
\node[endpoint] (UY) at (5, 0.6) {};
\node[above left, font=\small] at (LY) {\(L_Y=2\)};
\node[above right, font=\small] at (UY) {\(U_Y=L_Y+d_Y\)};

\node[anchor=east, font=\bfseries] at (-1, -0.6) {Candidates};

\node[blocked] (C0) at (0, -0.6) {\(\times\)};
\node[below, font=\footnotesize, text=red!80!black] at (0, -0.8) {\(t=0\)};

\node[blocked] (C1) at (1, -0.6) {\(\times\)};
\node[below, font=\footnotesize, text=red!80!black] at (1, -0.8) {\(t=1\)};

\node[candidate] (C2) at (2, -0.6) {\(\checkmark\)};
\node[below, font=\footnotesize, text=black!60] at (2, -0.8) {\(t=2=L_Y\)};

\draw[->, red!80, thick, shorten >=2pt, shorten <=2pt] (LX) -- (C0);
\node[blocklabel] at (0.2, 1.2) {Direct block \(L_X=0\)};

\draw[->, red!80, thick, dashed, shorten >=2pt, shorten <=2pt]
    (UX) .. controls (4, -0.2) and (1.5, -0.2) .. (C1.east);

\node[
    font=\scriptsize,
    text=red!80!black,
    align=left,
    anchor=west,
    fill=white,
    inner sep=2pt
] at (4.2, -0.7) {
    Shifted block for \(t=1\):\\
    \(1+d_Y=U_X\), hence \(U_X=d_Y+1\).
};

\end{tikzpicture}

\caption{Boundary configuration \(L_Y=2\) in the two-terminal endpoint lemma.
}
\label{fig:three-machine-boundary-case}
\end{figure}
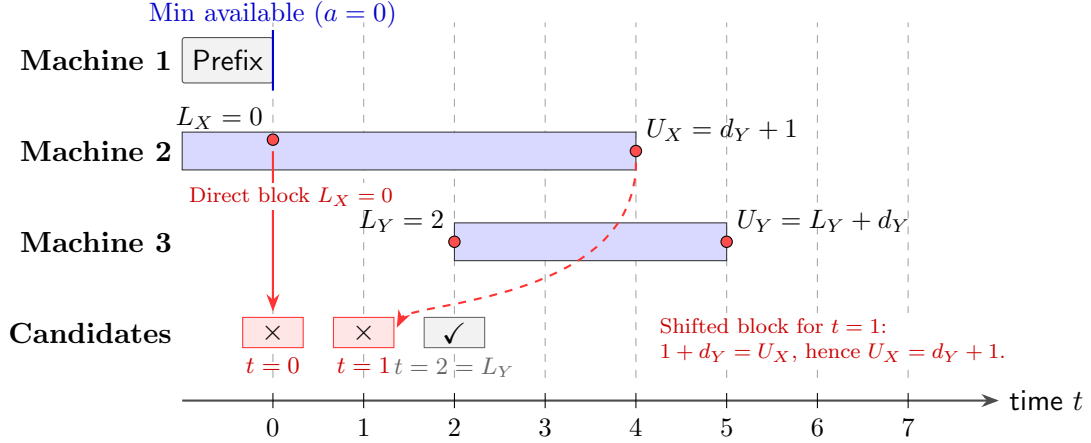

\subsection{Density estimates}

\begin{lemma}[Local blocking density]
\label{lem:local-blocking-density-three}
Consider one LS insertion on three machines. Shift time so that the selected
machine becomes available at time \(0\). Let the inserted job have offset
\(d=e-1\ge2\), and let \(S\) be its LS loading time. Let \(R(t)\) be the
resource profile of the partial schedule before this insertion. Then
\[
        \sum_{t=0}^{S-1}R(t)\ge2S-1 .
\]
Equivalently, before shifting, if the selected-machine availability is \(a\),
then
\[
        \sum_{t=a}^{S-1}R(t)\ge2(S-a)-1 .
\]
\end{lemma}

\begin{proof}
The case \(S=0\) is trivial. Assume \(S>0\). By
Lemma~\ref{lem:blocking-cover}, each candidate \(t=0,\ldots,S-1\) is blocked by
a direct or shifted server conflict:
\[
        t\in B
        \qquad\text{or}\qquad
        t+d\in B .
\]
By Lemma~\ref{lem:ls-terminal-property-prelim}, all previously scheduled jobs relevant to these
blocking conditions are terminal. Since one machine has completion time zero, at most
two such terminal jobs exist.

If no previously scheduled terminal job has a non-negative occupation slot or endpoint, then no
candidate \(t\ge0\) can be blocked, and candidate \(0\) is feasible. Hence this
case cannot occur when \(S>0\).

Suppose first that exactly one previously scheduled terminal job \(X\) is relevant. Then
\(L_X\le0\le U_X\). The only possible non-negative blocking positions are
\[
        0 \quad\text{if }L_X=0,\qquad U_X-d,\qquad U_X .
\]
The positions \(U_X-d\) and \(U_X\) differ by \(d\ge2\). Hence a consecutive
blocked interval starting at \(0\) has a length at most two. If \(S=1\), then
\(R(0)\ge1=2S-1\). If \(S=2\), candidates \(0\) and \(1\) must both be blocked,
which forces \(L_X=0\) and \(U_X-d=1\). Thus, \(R(0)\ge2\) and \(R(1)\ge1\),
giving \(R(0)+R(1)\ge3=2S-1\).

It remains to consider two relevant terminal jobs \(X\) and \(Y\), ordered as in
Lemma~\ref{lem:two-terminal-endpoint-three}. We first show that no slot in
\([0,S)\) has resource zero. Before \(Y\) starts, this follows from
Lemma~\ref{lem:two-terminal-endpoint-three}: if \(L_Y=1\), then \(X\) is active
at slot \(0\); if \(L_Y=2\), then \(L_X=0\) and \(U_X=d_Y+1\), so \(X\) is
active at slots \(0\) and \(1\). After \(Y\) starts, if both terminal jobs had
completed before a blocked candidate \(t<S\), then neither \(t\) nor \(t+d\)
could be an endpoint of \(X\) or \(Y\), contradicting the fact that \(t\) is
blocked. Therefore,
\[
        R(t)\ge1,\qquad t=0,\ldots,S-1 .
\]

We call a slot \(t\in[0,S)\) sparse if \(R(t)=1\). A sparse slot contains no server
operation; otherwise the same job would also be machine-active in that slot,
giving \(R(t)\ge2\). Hence a sparse slot must be shifted-blocked.

Every non-final sparse slot is followed by a slot of resource at least \(3\).
Before \(Y\) starts, the only possible sparse slot is \(0\) when \(L_Y=1\), or
\(1\) when \(L_Y=2\). In either case, the following slot contains the loading of
\(Y\), while \(X\) remains active; hence that slot has resource value at least \(3\).
After \(Y\) starts, a sparse slot has exactly one of \(X,Y\) active. Its shifted
blocker must be the unloading endpoint of that active terminal job. Since
\(d\ge2\), the same endpoint cannot block the next candidate. The other terminal
job has already completed, so the next candidate would be feasible unless the
sparse slot is \(S-1\).

Thus, every non-final sparse slot is charged injectively to a following slot of
resource at least \(3\), and at most the final sparse slot remains uncharged.
Relative to the target average of two resource units per slot, each sparse
slot creates a deficit of one, each charged successor contributes a surplus of
at least one, and at most one deficit is left uncharged. Hence
\[
        \sum_{t=0}^{S-1}R(t)\ge2S-1 .
\]
\end{proof}

\begin{lemma}[Local minimum-increase density]
\label{lem:local-minimum-increase-density-three}
Consider one LS insertion on three machines. Let \(a\) and \(a'\) be the minimum
machine completion times before and after the insertion. Let \(R_{\rm new}(t)\)
be the aggregate resource profile after the insertion. Then
\[
        \sum_{t=a}^{a'-1}R_{\rm new}(t)\ge2(a'-a).
\]
\end{lemma}

\begin{proof}
The result is trivial if \(a'=a\). Suppose \(a'>a\), shift time so that \(a=0\),
and put \(T=a'>0\). The selected machine was the unique minimum before the
insertion; otherwise the minimum would remain zero after the insertion.

Let \(Z\) be the inserted job, and let its loading time be \(S_Z\). The two
other machines end with terminal jobs, say \(X\) and \(Y\), with completion
times at least \(T\).

First suppose \(S_Z<T\). Before the insertion, the partial schedule satisfies
\[
        \sum_{t=0}^{S_Z-1}R_{\rm old}(t)\ge2S_Z-1
\]
by Lemma~\ref{lem:local-blocking-density-three}. In the new schedule, on the
interval \([S_Z,T)\), the job \(Z\) is active. Also \(X\) is active throughout
\([0,T)\), by Lemma~\ref{lem:two-terminal-endpoint-three}. Hence every slot in
\([S_Z,T)\) has at least two units of resource, and the loading slot \(S_Z\) has
one additional server unit. Therefore, 
\[
        \sum_{t=S_Z}^{T-1}R_{\rm new}(t)\ge 2(T-S_Z)+1 .
\]
Combining the two estimates gives
\[
        \sum_{t=0}^{T-1}R_{\rm new}(t)\ge 2T .
\]

Now suppose \(S_Z\ge T\). Then every candidate \(t=0,\ldots,T-1\) was infeasible
for \(Z\). The sparse-slot charging used in
Lemma~\ref{lem:local-blocking-density-three} remains inside \([0,T)\), except
possibly for a sparse slot at \(T-1\): every sparse slot \(t<T-1\) is charged to
\(t+1\), which also lies in \([0,T)\). The final slot \(T-1\) cannot be sparse
in the new schedule. Indeed, some job attains the new minimum completion time
\(T\); this job unloads in slot \(T-1\) and is machine-active there, so
\(R_{\rm new}(T-1)\ge2\). Hence no uncharged sparse deficit remains, and
\[
        \sum_{t=0}^{T-1}R_{\rm new}(t)\ge2T .
\]

Returning to the original time scale gives the claim.
\end{proof}

\begin{lemma}[Minimum-frontier density]
\label{lem:minimum-frontier-density-three}
For any LS prefix on three machines, let
\[
        a=\min_i c_i
\]
be the minimum current machine completion time, and let \(R(t)\) be the resource
profile of that prefix. Then
\[
        \sum_{t=0}^{a-1}R(t)\ge2a .
\]
\end{lemma}

\begin{proof}
We argue by induction over the LS insertions. The empty prefix is immediate.
Consider one insertion, and let \(a_{\rm old}\) and \(a_{\rm new}\) be the
minimum machine completion times before and after it. If
\(a_{\rm new}=a_{\rm old}\), the interval \([0,a_{\rm new})\) is unchanged and
the resource profile cannot decrease.

If \(a_{\rm new}>a_{\rm old}\), the induction hypothesis gives
\[
        \sum_{t=0}^{a_{\rm old}-1}R_{\rm old}(t)\ge2a_{\rm old}.
\]
Adding a job cannot decrease \(R(t)\) on earlier slots, and
Lemma~\ref{lem:local-minimum-increase-density-three} gives
\[
        \sum_{t=a_{\rm old}}^{a_{\rm new}-1}R_{\rm new}(t)
        \ge2(a_{\rm new}-a_{\rm old}).
\]
Combining the two inequalities gives
\[
        \sum_{t=0}^{a_{\rm new}-1}R_{\rm new}(t)\ge2a_{\rm new}.
\]
\end{proof}

\subsection{The approximation guarantee}

\begin{theorem}
\label{thm:ls-three-five-halves}
For problem 
\[
        P3,S1\mid s_j=t_j=1\mid C_{\max},
\]
LS satisfies
\[
        C_{\max}^{\mathrm{LS}}\le \frac52 C^* .
\]
\end{theorem}

\begin{proof}
Let \(J^*\) be a job attaining the makespan in the LS schedule, with loading
time \(S^*\) and execution length \(e^*\). Thus, 
\[
        C_{\max}^{\mathrm{LS}}=S^*+e^* .
\]
Let \(a\) be the selected-machine availability immediately before \(J^*\) is
inserted, and let \(R(t)\) be the resource profile of the prefix before this
insertion. By Lemmas~\ref{lem:minimum-frontier-density-three} and
\ref{lem:local-blocking-density-three}, we get 
\[
        \sum_{t=0}^{S^*-1}R(t)
        \ge
        2a+2(S^*-a)-1
        =
        2S^*-1 .
\]

In the final schedule, job \(J^*\) contributes \(e^*\) units of machine occupation
and two server operations, all in slots from \(S^*\) onward. These \(e^*+2\)
resource units are disjoint from the prefix resource counted on
\([0,S^*)\). Since the total resource in the final schedule is \(E+2n\), we have 
\[
        E+2n
        \ge
        (2S^*-1)+(e^*+2)
        =
        2S^*+e^*+1 .
\]
In particular,
\[
        E+2n\ge2S^*+e^* .
\]
Since \(e^*\le e_{\max}\),
\[
        E+2n+e_{\max}
        \ge
        2S^*+2e^*
        =
        2C_{\max}^{\mathrm{LS}}.
\]
Therefore, we have
\[
        C_{\max}^{\mathrm{LS}}
        \le
        \frac{E+2n+e_{\max}}2 .
\]
By Lemma~\ref{lem:lower-bounds},
\[
        E\le3C^*,
        \qquad
        2n\le C^*,
        \qquad
        e_{\max}\le C^* .
\]
Substituting, we get
\[
        C_{\max}^{\mathrm{LS}}
        \le
        \frac{3C^*+C^*+C^*}{2}
        =
        \frac52 C^* .
\]
\end{proof}


\section{General Performance of List Scheduling}
\label{sec:ls-general-performance}

We consider the problem
$P,S1\mid s_j=t_j=1\mid C_{\max},$ for $
 m\ge 3$ and $n\ge 1 .$
For each job \(J_j\), define its execution length by
$e_j=p_j+2 .$

Thus,  \(e_j\in\mathbb Z_{\ge 3}\). If \(J_j\) starts loading at integer time
\(S_j\), then its unloading slot and completion time are
\[
U_j=S_j+e_j-1,
\qquad
C_j=S_j+e_j .
\]
A server operation at integer time \(t\) occupies the unit slot \([t,t+1)\).
The common server can process at most one loading or unloading operation in any
slot. Hence the server slots
\[
S_j,\quad U_j,\qquad j=1,\ldots,n,
\]
are pairwise distinct.

Let
\[
E=\sum_{j=1}^{n}e_j,
\qquad
e_{\max}=\max_{1\le j\le n}e_j .
\]
List scheduling (LS) considers the jobs in the prescribed list order. When job
\(J_k\) is considered, LS selects a machine with minimum current completion time
and schedules \(J_k\) at the earliest integer loading time for which both its
loading and unloading server slots are free.

We shall use the following lower bounds. The first three are standard. The last
one is a startup and termination bound specific to the loading-unloading setting.

\begin{lemma}
\label{lem:ls-lower-bounds}
Every feasible integral schedule satisfies
\[
C^*\ge \frac{E}{m},
\qquad
C^*\ge 2n,
\qquad
C^*\ge e_{\max},
\qquad
C^*\ge \frac{E+2(m-1)}{m}.
\]
\end{lemma}

\begin{proof}
The total machine occupation is \(E\), and hence \(C^*\ge E/m\). The server
performs exactly \(2n\) unit operations, and hence \(C^*\ge 2n\). Moreover, each job
occupies a machine continuously for \(e_j\) time units, so \(C^*\ge e_{\max}\).

It remains to prove the last bound. We may shift any feasible schedule so that
its first server operation starts at time zero. In the first slot \([0,1)\), at
most one job can be loaded, and hence at most one machine can be occupied. Thus,
at least \(m-1\) machine-capacity units are idle in this slot.

Consider the last slot \([C^*-1,C^*)\) of an optimal integral schedule. Any job
active in this slot must complete at time \(C^*\), and hence must use the server
for unloading in slot \(C^*-1\). Since the server can unload at most one job in
this slot, at most one machine can be occupied in the last slot. Thus at least
\(m-1\) further machine-capacity units are idle in this slot.

Because \(e_j\ge 3\) and \(n\ge1\), we have \(C^*\ge3\), so the first and last
slots are distinct. Therefore, the total machine occupation satisfies
\[
E\le mC^*-2(m-1),
\]
which gives
\[
C^*\ge \frac{E+2(m-1)}{m}.
\]
\end{proof}

\subsection{A machine-sequence estimate}
\label{subsec:ls-single-chain-estimate}

Fix the partial LS schedule immediately before job \(J_\ell\) is inserted, and
write
\[
        h=\ell-1.
\]
For each machine \(i\), let
\[
        J_{i,1},J_{i,2},\ldots,J_{i,r_i}
\]
be the jobs among \(J_1,\ldots,J_h\) assigned to that machine, listed in their
processing order. Denote their loading times, execution lengths, and completion
times by
\[
        S_{i,a},\qquad e_{i,a},\qquad
        C_{i,a}=S_{i,a}+e_{i,a},
        \qquad a=1,\ldots,r_i.
\]
Set \(C_{i,0}=0\), and define
\[
        L_i=\sum_{a=1}^{r_i}e_{i,a}.
\]
Then
\[
        \sum_{i=1}^m L_i=\sum_{j<\ell}e_j,
        \qquad
        \sum_{i=1}^m r_i=h.
\]

Let \(c_i=C_{i,r_i}\) if \(r_i>0\), and let \(c_i=0\) otherwise. Suppose that job
\(J_\ell\) were assigned to machine \(i\), and let \(T_i\) be its earliest
server-feasible loading time not earlier than \(c_i\). The idle slots on machine
\(i\) before the loading of job  \(J_\ell\) are
\[
\begin{aligned}
        \mathcal W_i
        ={}&
        \bigcup_{a=1}^{r_i}
        [C_{i,a-1},S_{i,a})_{\mathbb Z}
        \\
        &{}\cup
        [c_i,T_i)_{\mathbb Z}.
\end{aligned}
\]
These intervals are pairwise disjoint. Let
\[
        W_i=|\mathcal W_i|.
\]
Since the interval from time zero to \(T_i\) consists of \(L_i\) execution slots
and \(W_i\) idle slots,
\[
        T_i=L_i+W_i
\]
holds. The essential observation is that none of these idle slots can be attributed to
a server operation of a job already assigned to machine \(i\).

\begin{lemma}
\label{lem:ls-chain-waiting}
For every machine \(i\), we have
\[
        W_i\le 4(h-r_i).
\]
\end{lemma}

\begin{proof}
Associate each slot in \(\mathcal W_i\) with one server operation that makes the
corresponding loading candidate infeasible.

Consider first a slot
\[
        t\in[C_{i,a-1},S_{i,a})_{\mathbb Z}
\]
before an actual job \(J_{i,a}\) on machine \(i\). At the time when
\(J_{i,a}\) was inserted, machine \(i\) was available from
\(C_{i,a-1}\) onward. Since \(S_{i,a}\) is the earliest server-feasible loading
time, the candidate \(t\) was rejected. Thus, either \(t\) or
\[
        t+d_{i,a},
        \qquad d_{i,a}=e_{i,a}-1,
\]
was occupied by a server operation that had already been scheduled. Choose one
such operation.

For a slot
\[
        t\in[c_i,T_i)_{\mathbb Z},
\]
choose in the same way one server operation from the partial schedule
\(J_1,\ldots,J_h\) that blocks the loading of \(J_\ell\).

Every chosen operation belongs to a job assigned to a machine other than
\(i\). Indeed, before \(J_{i,a}\) is loaded, all earlier jobs on machine \(i\)
have been completed by time \(C_{i,a-1}\le t\), so their loading and unloading
operations occur strictly before \(t\). Jobs later in the sequence on machine
\(i\) had not yet been inserted. Likewise, before the possible loading of job 
\(J_\ell\), all \(r_i\) jobs already assigned to machine \(i\) have been completed
by time \(c_i\le t\). Hence none of their server operations can block the
candidate.

The \(h-r_i\) jobs assigned to the other machines contribute exactly
\[
        2(h-r_i)
\]
server operations. Fix one of these operations and let \(s\) be its time slot.
For a given job \(J\), with \(d_J=e_J-1\), the operation at \(s\) can block
only the two candidate loading times
\[
        t=s
        \qquad\text{and}\qquad
        t=s-d_J.
\]

Moreover, the same operation cannot block candidates belonging to two different
jobs on machine \(i\). Suppose that it blocks a candidate \(t\) before a job
\(J\). In the direct case, \(s=t\), and \(J\) is loaded no earlier than
\(s+1\). In the shifted case, \(s=t+d_J\), and the completion time of \(J\) is
at least
\[
        t+1+e_J=s+2.
\]
Thus, after this rejection, machine \(i\) remains occupied beyond time \(s\).
Every loading candidate for a later job on the same machine is therefore larger
than \(s\), whereas the operation at \(s\) can block only \(s\) or a time
smaller than \(s\).

Consequently, each server operation of a job assigned to another machine is
associated with at most two slots of \(\mathcal W_i\). Therefore, inequality
\[
        W_i\le 2\cdot 2(h-r_i)=4(h-r_i)
\]
holds.
\end{proof}

\begin{lemma}
\label{lem:ls-critical-start}
Let \(J_\ell\) be a critical job in the LS schedule. Then
\[
        S_\ell
        \le
        \frac{\sum_{j<\ell}e_j}{m}
        +
        4\left(1-\frac1m\right)(\ell-1).
\]
Consequently,
\[
        C_{\max}^{\mathrm{LS}}
        \le
        \frac{E}{m}
        +
        4\left(1-\frac1m\right)(\ell-1)
        +
        \left(1-\frac1m\right)e_\ell.
\]
\end{lemma}

\begin{proof}
Let \(c_i\) be the completion time of machine \(i\) immediately before
\(J_\ell\) is inserted, and let
\[
        a_\ell=\min_{1\le i\le m}c_i.
\]
For the server schedule generated by the jobs  \(J_1,\ldots,J_h\), define
\[
        F_\ell(a)
        =
        \min\left\{
        t\in\mathbb Z_{\ge a}:
        t\ \text{and}\ t+d_\ell\ \text{are free server slots}
        \right\}.
\]
The function \(F_\ell\) is non-decreasing. Since Algorithm LS selects a machine with
minimum completion time, we have 
\[
        S_\ell=F_\ell(a_\ell).
\]
If \(J_\ell\) were assigned to machine \(i\), its earliest feasible loading
time would be
\[
        T_i=F_\ell(c_i).
\]
Because \(a_\ell\le c_i\),
\[
        S_\ell\le T_i
        \qquad
        \text{for every }i.
\]

By Lemma~\ref{lem:ls-chain-waiting},
\[
        T_i=L_i+W_i\le L_i+4(h-r_i).
\]
Therefore,
\[
\begin{aligned}
        S_\ell
        &\le
        \min_{1\le i\le m}
        \left\{L_i+4(h-r_i)\right\} \\
        &\le
        \frac1m\sum_{i=1}^m
        \left\{L_i+4(h-r_i)\right\} \\
        &=
        \frac{\sum_{j<\ell}e_j}{m}
        +
        4\left(1-\frac1m\right)h.
\end{aligned}
\]
Using \(h=\ell-1\) gives the first inequality. Since
\[
        \sum_{j<\ell}e_j\le E-e_\ell,
\]
we further obtain
\[
\begin{aligned}
        C_{\max}^{\mathrm{LS}}
        &=S_\ell+e_\ell \\
        &\le
        \frac{E-e_\ell}{m}
        +
        4\left(1-\frac1m\right)(\ell-1)
        +e_\ell \\
        &=
        \frac{E}{m}
        +
        4\left(1-\frac1m\right)(\ell-1)
        +
        \left(1-\frac1m\right)e_\ell.
\end{aligned}
\]
\end{proof}

\begin{theorem}
\label{thm:ls-general-bound}
For problem 
\[
        Pm,S1\mid s_j=t_j=1,\ p_j\ge1\mid C_{\max},
\]
with \(m\ge3\), Algorithm LS satisfies the inequality 
\[
        C_{\max}^{\mathrm{LS}}
        \le
        \left(
        4-\frac3m-\frac{2(m-1)}{mn}
        \right)C^*.
\]
\end{theorem}

\begin{proof}
Let \(J_\ell\) be a critical job. By Lemma~\ref{lem:ls-critical-start}, we obtain 
\[
        C_{\max}^{\mathrm{LS}}
        \le
        \frac{E}{m}
        +
        4\left(1-\frac1m\right)(\ell-1)
        +
        \left(1-\frac1m\right)e_\ell.
\]
Since \(\ell\le n\) and \(e_\ell\le e_{\max}\) we get,
\[
        C_{\max}^{\mathrm{LS}}
        \le
        \frac{E}{m}
        +
        4\left(1-\frac1m\right)(n-1)
        +
        \left(1-\frac1m\right)e_{\max}.
\]
Using
\[
        \frac{E}{m}\le C^*,
        \qquad
        2n\le C^*,
        \qquad
        e_{\max}\le C^*,
\]
and
\[
\begin{aligned}
        4\left(1-\frac1m\right)(n-1)
        &=
        2\left(1-\frac1m\right)
        \left(1-\frac1n\right)(2n) \\
        &\le
        2\left(1-\frac1m\right)
        \left(1-\frac1n\right)C^*,
\end{aligned}
\]
we obtain
\[
\begin{aligned}
        C_{\max}^{\mathrm{LS}}
        &\le
        \left[
        1
        +
        2\left(1-\frac1m\right)
        \left(1-\frac1n\right)
        +
        \left(1-\frac1m\right)
        \right]C^* \\
        &=
        \left(
        4-\frac3m-\frac{2(m-1)}{mn}
        \right)C^*.
\end{aligned}
\]
\end{proof}

\section{Performance of the LPT Rule for Multiple Machines}
\label{sec:lpt-general-m}

We now analyse the LPT rule. Since \(e_j=p_j+2\), ordering the jobs by
non-increasing processing time is equivalent to ordering them by nonincreasing
execution length. Hence the jobs are indexed so that
\[
        e_1\ge e_2\ge\cdots\ge e_n.
\]
They are then scheduled by the same insertion rule used by Algorithm LS.

\begin{theorem}
\label{thm:lpt-general-m}
For problem 
\[
        Pm,S1\mid s_j=t_j=1,\ p_j\ge1\mid C_{\max},
\]
with \(n\ge m\ge3\), Algorithm LPT satisfies the inequality 
\[
        C_{\max}^{\mathrm{LPT}}
        \le
        \left(
        3-\frac2m+
        \frac{(m-1)(m-2)}{mn}
        \right)C^*.
\]
\end{theorem}

\begin{proof}
Let \(J_\ell\) be a critical job, and set
\[
        \alpha=1-\frac1m.
\]
Lemma~\ref{lem:ls-critical-start} gives
\[
        C_{\max}^{\mathrm{LPT}}
        \le
        \frac{E}{m}
        +4\alpha(\ell-1)
        +\alpha e_\ell.
\]
We distinguish two cases.

\paragraph{Case 1: \(\ell<m\).}
Using
\[
        \frac{E}{m}\le C^*,
        \qquad
        2n\le C^*,
        \qquad
        e_\ell\le e_{\max}\le C^*,
\]
we obtain
\[
        \frac{C_{\max}^{\mathrm{LPT}}}{C^*}
        \le
        1+\alpha+
        \frac{2\alpha(\ell-1)}{n}.
\]
Since \(\ell\le m-1\) and \(n\ge m\), we have 
\[
        2(\ell-1)\le2m-4\le n+m-2.
\]
Therefore, we get 
\[
        1+\alpha+
        \frac{2\alpha(\ell-1)}{n}
        \le
        1+2\alpha+
        \frac{\alpha(m-2)}{n}
        =
        3-\frac2m+
        \frac{(m-1)(m-2)}{mn}.
\]

\paragraph{Case 2: \(\ell\ge m\).}
The LPT order implies
\[
        E\ge\sum_{j=1}^{\ell}e_j\ge\ell e_\ell,
\]
so that
\[
        e_\ell\le\frac{E}{\ell}\le\frac{m}{\ell}C^*.
\]
Together with \(E/m\le C^*\) and \(2n\le C^*\), this yields
\[
        \frac{C_{\max}^{\mathrm{LPT}}}{C^*}
        \le
        1+
        \frac{2\alpha(\ell-1)}{n}
        +
        \frac{m-1}{\ell}.
\]
Define
\[
        f(x)
        =
        1+
        \frac{2\alpha(x-1)}{n}
        +
        \frac{m-1}{x}.
\]
Since
\[
        f''(x)=\frac{2(m-1)}{x^3}>0,
\]
\(f\) is convex on \((0,\infty)\). Its maximum on \([m,n]\) is therefore
attained at an endpoint. Moreover,
\[
        f(n)-f(m)
        =
        \frac{(m-1)(n-m)}{mn}
        \ge0.
\]
Hence \(f(\ell)\le f(n)\), and
\[
\begin{aligned}
        f(n)
        &=
        1+
        2\left(1-\frac1m\right)\frac{n-1}{n}
        +
        \frac{m-1}{n} \\
        &=
        3-\frac2m+
        \frac{(m-1)(m-2)}{mn}.
\end{aligned}
\]
This completes the proof.
\end{proof}

\begin{corollary}
\label{cor:lpt-asymptotic}
For every fixed \(m\ge3\), we have 
\[
        \limsup_{n\to\infty}
        \frac{C_{\max}^{\mathrm{LPT}}}{C^*}
        \le
        3-\frac2m.
\]
\end{corollary}

\begin{proof}
The result follows from Theorem~\ref{thm:lpt-general-m}, since
\[
        \frac{(m-1)(m-2)}{mn}\longrightarrow0
\]
as \(n\to\infty\).
\end{proof}

\begin{remark}
For \(m=3\), Theorem~\ref{thm:lpt-general-m} gives
\[
        C_{\max}^{\mathrm{LPT}}
        \le
        \left(\frac73+\frac{2}{3n}\right)C^*.
\]
This coefficient is no larger than \(5/2\) for \(n\ge4\). For \(n=3\), the
bound in Theorem~\ref{thm:ls-three-five-halves} is stronger.
\end{remark}

\section{Conclusion}
\label{sec:conclusion}

This paper studied a parallel machine scheduling problem with a single common server that performs both loading and unloading operations. Each job requires a unit-time loading operation, non-preemptive processing on one of $m$ identical machines, and a unit-time unloading operation by the same server. The resulting model differs from the loading-only common-server setting because every job creates two server operations separated by an exact delay determined by its processing time.

We first established a complexity result for the variable-machine case. By adapting the modular construction of Kravchenko and Werner to the loading-unloading setting, we proved that the decision version of problem 
$ P,S1\mid s_j=t_j=1\mid C_{\max}$ is strongly NP-complete when $m$ is part of the input. Equivalently, the corresponding optimization problem is strongly NP-hard. The proof shows that the second server operation cannot be treated as a minor extension of the loading-only model; the construction must explicitly separate loading and unloading residues to maintain synchronisation.

We then analysed list-based schedules. For three machines, the blocking structure after the minimum machine completion time admits a sharper treatment, leading to the bound
$$ C_{\max}^{\rm LS}\le \frac52 C^* .$$ Since Algorithm LPT generates a particular list schedule, the same guarantee holds for Algorithm LPT on three machines. For arbitrary fixed $m\ge3$, we derived a general upper bound for List Scheduling. In particular, this bound tends to $4-3/m$ as the number of jobs grows. For Algorithm LPT, we proved the finite-instance bound
$$C_{\max}^{\rm LPT}
        \le
        \left(3-\frac2m+\frac{(m-1)(m-2)}{mn}\right)C^*.$$
In particular, for each fixed \(m\), this guarantee tends to \(3-2/m\) as
\(n\to\infty\). These results give the first general performance guarantees for list-based rules in the unit loading-unloading model beyond the two-machine case.

Several directions remain open. A first one is to sharpen the general List
Scheduling bound. Computational evidence suggests that the worst-case ratio may
be strictly below the guarantee proved here, but closing this gap would require
a finer analysis. A second
direction is to extend the approximation analysis to models with arbitrary
loading and unloading times. Finally, the techniques developed in this paper may be useful for variants with
release dates, separate loading and unloading servers, or additional operational
constraints arising in industrial applications.
\bibliographystyle{elsarticle-harv}

\begin{thebibliography}{99}


\bibitem{Graham1966}
Graham, R. L. (1966).
Bounds for certain multiprocessing anomalies.
\emph{Bell System Technical Journal}, 45(9), 1563--1581.
doi:10.1002/j.1538-7305.1966.tb01709.x.

\bibitem{Graham1969}
Graham, R. L. (1969).
Bounds on multiprocessing timing anomalies.
\emph{SIAM Journal on Applied Mathematics}, 17(2), 416--429.
doi:10.1137/0117039.

\bibitem{AllahverdiNgChengKovalyov2008}
Allahverdi, A., Ng, C. T., Cheng, T. C. E., and Kovalyov, M. Y. (2008).
A survey of scheduling problems with setup times or costs.
\emph{European Journal of Operational Research}, 187(3), 985--1032.
doi:10.1016/j.ejor.2006.06.060.

\bibitem{HallPottsSriskandarajah2000}
Hall, N. G., Potts, C. N., and Sriskandarajah, C. (2000).
Parallel machine scheduling with a common server.
\emph{Discrete Applied Mathematics}, 102(3), 223--243.
doi:10.1016/S0166-218X(99)00206-1.

\bibitem{KravchenkoWerner1997}
Kravchenko, S. A., and Werner, F. (1997).
Parallel machine scheduling problems with a single server.
\emph{Mathematical and Computer Modelling}, 26(12), 1--11.
doi:10.1016/S0895-7177(97)00236-7.

\bibitem{BruckerDhaenensFlipoKnustKravchenkoWerner2002}
Brucker, P., Dhaenens-Flipo, C., Knust, S., Kravchenko, S. A., and Werner, F. (2002).
Complexity results for parallel machine problems with a single server.
\emph{Journal of Scheduling}, 5(6), 429--457.
doi:10.1002/jos.120.

\bibitem{JiangWangZhou2014}
Jiang, Y., Wang, H., and Zhou, P. (2014).
An optimal preemptive algorithm for the single-server parallel machine scheduling with loading and unloading times.
\emph{Asia-Pacific Journal of Operational Research}, 31(5), 1450039.
doi:10.1142/S0217595914500390.

\bibitem{JiangZhangHuDongJi2015}
Jiang, Y., Zhang, Q., Hu, J., Dong, J., and Ji, M. (2015).
Single-server parallel-machine scheduling with loading and unloading times.
\emph{Journal of Combinatorial Optimization}, 30(2), 201--213.
doi:10.1007/s10878-014-9727-z.

\bibitem{ElidrissiBenmansourHasaniWerner2024}
Elidrissi, A., Benmansour, R., Hasani, K., and Werner, F. (2024).
Minimizing the makespan on two parallel machines with a common server in charge of loading and unloading operations.
\emph{Computers \& Operations Research}, 167, 106638.
doi:10.1016/j.cor.2024.106638.

\bibitem{KhatamiSalehipourCheng2020}
Khatami, M., Salehipour, A., and Cheng, T. C. E. (2020).
Coupled task scheduling with exact delays: Literature review and models.
\emph{European Journal of Operational Research}, 282(1), 19--39.
doi:10.1016/j.ejor.2019.08.045.

\bibitem{ZhenLiangZhugeLeeChew2017}
Zhen, L., Liang, Z., Zhuge, D., Lee, L. H., and Chew, E. P. (2017).
Daily berth planning in a tidal port with channel flow control.
\emph{Transportation Research Part B: Methodological}, 106, 193--217.
doi:10.1016/j.trb.2017.10.008.

\bibitem{LiuLiShengWang2021}
Liu, B., Li, Z.-C., Sheng, D., and Wang, Y. (2021).
Integrated planning of berth allocation and vessel sequencing in a seaport with one-way navigation channel.
\emph{Transportation Research Part B: Methodological}, 143, 23--47.
doi:10.1016/j.trb.2020.10.010.

\bibitem{ZhangZheng2020}
Zhang, B., and Zheng, Z. (2020).
Model and algorithm for vessel scheduling through a one-way tidal channel.
\emph{Journal of Waterway, Port, Coastal, and Ocean Engineering}, 146(1), 04019032.
doi:10.1061/(ASCE)WW.1943-5460.0000545.

\end{thebibliography}

\end{document}